\documentclass[12pt]{amsart}
\usepackage{etex}
\usepackage[english]{babel}
\usepackage{latexsym}
\usepackage{amsbsy}
\usepackage{amssymb}
\usepackage{amsfonts, amsmath}
\usepackage{mathtools}
\usepackage{csquotes}
\usepackage{pst-tree}
\usepackage{multirow}
\usepackage{blkarray}
\usepackage{tabularx}
\usepackage{arydshln,leftidx,mathtools}
\usepackage{ifthen}
\usepackage{tikz}
\usepackage{tikz-cd}
\usepackage{float}
\usepackage{picins}
\usepackage{amsbsy}
\usetikzlibrary{matrix}
\usetikzlibrary{calc}
\usetikzlibrary{fit}
\usepackage{amsmath,amsfonts,amssymb}
\usepackage{mathrsfs,eurosym}
\usepackage{pst-plot,pst-eucl}
\usepackage{amsfonts}
\usepackage{amssymb}
\usepackage{amsmath}

\usepackage{latexsym}
\usepackage{amsbsy}
\usepackage{alltt}
\usepackage{fontenc}
\usepackage{newlfont}
\usepackage{xlop}
\usepackage{graphicx}
\usepackage{yhmath}
\usepackage{mathdots}

\usepackage{MnSymbol}
\usepackage{amsthm}
\usepackage{url}
\usepackage{enumerate}

\newtheorem{thrm}{Theorem}[section]
\newtheorem{lem}[thrm]{Lemma}

\newtheorem{cor}[thrm]{Corollary}
\theoremstyle{definition}
\newtheorem{definition}[thrm]{Definition}
\newtheorem{remark}[thrm]{Remark}
\newtheorem{example}[thrm]{Example}
\numberwithin{equation}{section}

\usepackage{cite}
\numberwithin{equation}{section}

\newcommand{\e}{{\mathrm{e}}}

\title[Minimal polynomials and the matrix exponential function]{An explicit expression for the minimal polynomial of the Kronecker product of matrices. Explicit
formulas for matrix logarithm and matrix exponential}
\author{\bf M. MOU\c{C}OUF}
\date{}

\keywords{$\mathcal{P}$-Canonical form, Powers, Linear recurrence sequences, Matrix exponential, Logarithm of a matrix, Kronecker product}

\begin{document}
\maketitle
\begin{center}
{\footnotesize Department of Mathematics, Faculty of Science, Chouaib Doukkali University,\\Morocco\\
Email: moucouf@hotmail.com}
\end{center}
\begin{abstract}
Using $\mathcal{P}$-canonical forms of matrices, we derive the minimal polynomial of the Kronecker product of a given family of matrices in terms of the minimal polynomials of these matrices. This, allows us to prove that the product $\prod\limits_{i=1}^{m}L(P_{i})$, $L(P_{i})$ is the set of linear recurrence sequences over a field $F$ with characteristic polynomial $P_{i}$, is equal to $L(P)$ where $P$ is the minimal polynomial of the Kronecker product of the companion matrices of $P_{i}$, $1\leq i\leq m$. Also, we show how we deduce from the $\mathcal{P}$-canonical form of an arbitrary complex matrix $A$, the $\mathcal{P}$-canonical form of the matrix function $\e^{tA}$ and a logarithm of $A$.
\end{abstract}
%
\section{Introduction}
\label{sec1}
%
This paper is devoted to some applications of the main results of the paper~\cite{Mouc2} and the $\mathcal{P}$-canonical forms of matrices presented in~\cite{Mouc1}. The paper is organized as follows. In section~$2$,  we show that an important application of the $\mathcal{P}$-canonical forms of matrices is to derive the minimal polynomial of the Kronecker product of a given family of matrices in terms of the minimal polynomials of these matrices. This result allows us to prove that, if $P_{1},\ldots, P_{m}$ are monic polynomials over a field $F$, then the product $\prod\limits_{i=1}^{m}L(P_{i})$ is equal to $L(P)$ where $P$ is the minimal polynomial of the Kronecker product of the companion matrices of $P_{i}$, $1\leq i\leq m$. In section~$3$, we investigate the $\mathcal{P}$-canonical form of the matrix function $\e^{tA}$ of a an arbitrary complex matrix $A$ and, we show also how we can simply obtain the matrix logarithm $\log(A)$ and its $\mathcal{P}$-canonical form if $A$ is nonsingular. More precisely, we describe a relationship between these $\mathcal{P}$-canonical forms which is helpful for determining two of them if the other one is known. Finally, in section~4 some theoretical and numerical examples are presented to verify the
theoretical results.
%
\\\indent Through this paper, we use the following notations which are similar to those used in the papers~\cite{Mouc1,Mouc2}.
\begin{itemize}
\item $F$ is an arbitrary field.
\item $J_{s}(\alpha)$ denotes the Jordan block of order $s$ associated to $\alpha \in F$.
\item $[a_{1},\ldots,a_{n}]_{n}$ denotes the semicirculant matrix whose first row is $(a_{1},\ldots,a_{n})$.
\item The symbol $\sim$ denotes the similarity of matrices. It is well known that if $A\sim B$ and $C\sim D$ then $A\otimes B\sim C\otimes D$.
\item For all nonzero element $\lambda$ of $F$, $\pmb{\lambda}$ denotes the geometric sequence $(\lambda^{k})_{k\geq 0}$.
\item For all nonegative integer $n$, $\pmb{0}_{n}$ denotes the sequence $(\delta_{n,k})_{k\geq 0}$ where $\delta_{n,k}$ is the Kronecker symbol.
\item For all nonegative integer $i$, $\Lambda_{i}$ denotes the sequence $(\binom{k}{i})_{k\geq 0}$.
\item For all nonzero element $\lambda$ of $F$ and all positive integer $s$, $\langle\lambda\rangle_{s}$ denotes the subspace of $\mathcal{C}_{F}$ spanned by $\{\pmb{\lambda}\Lambda_{0},\ldots,\pmb{\lambda}\Lambda_{s-1}\}$, where $\mathcal{C}_{F}$ is the set of all linear recurrence sequences over $F$.
\item For all positive integer $s$, $\langle 0\rangle_{s}$ denotes the subspace of $\mathcal{C}_{F}$ spanned by $\{\pmb{0}_{0},\ldots,\pmb{0}_{s-1}\}$.
\item $\langle\lambda\rangle_{0}=\langle0\rangle_{0}=0$ the zero subspace of $\mathcal{C}_{F}$.
\item $\Gamma$ denotes the sequence $(0,1,2,\ldots)$.
\item $\mathcal{S}^{\ast}=\{\pmb{\lambda}=(\lambda^{k})_{k\geqslant 0}/\lambda\in F, \lambda\neq 0\}$ denotes the set of all nonzero geometric sequences.
\item $F_{\mathcal{S}^{\ast}}$ denotes the $F$-vector spaces spanned by $\mathcal{S}^{\ast}$.
\item $\mathcal{T}=\{\Lambda_{n}/ n\in \mathbb{N}\}$.
\item  $\mathcal{H}=(\Gamma^{n})_{n\geqslant 0}$.
\end{itemize}
Let $A$ be a square matrix over a field $F$ and let $\pmb{A}$ denote the sequence $(A^{k})_{k\geq0}$. Then there exist matrices $\mathcal{V}_{0},\ldots, \mathcal{V}_{n}$ with coefficients in $F$ and matrices $\mathcal{A}_{0},\ldots, \mathcal{A}_{l}$ with coefficients in $F_{\mathcal{S}^{\ast}}$ such that
\begin{equation*}
\pmb{A}= N(A) + \mathcal{A}_{0}\Lambda_{0} + \cdots + \mathcal{A}_{l}\Lambda_{l}
\end{equation*}
If the characteristic of $F$ is $0$, then there exist matrices $\mathcal{A'}_{0},\ldots, \mathcal{A'}_{l}$ with coefficients in $F_{\mathcal{S}^{\ast}}$ such that
\begin{equation*}
\pmb{A}=N(A) + \mathcal{A'}_{0}\Gamma^{0} + \cdots + \mathcal{A'}_{l}\Gamma^{l}
\end{equation*}
where $$N(A)=\mathcal{V}_{0}\pmb{0}_{0}+\cdots+\mathcal{V}_{n}\pmb{0}_{n}.$$
The matrices $\mathcal{V}_{0},\ldots, \mathcal{V}_{n}, \mathcal{A}_{0},\ldots, \mathcal{A}_{l}, \mathcal{A'}_{0},\ldots, \mathcal{A'}_{l}$ are uniquely determined by $A$.
\\We say that
\begin{itemize}
\item $N(A)$ is the non-geometric part of $A$.
\item $\pmb{A}-N(A)$ is the geometric part of $A$.
\item $\pmb{A}= N(A) + \mathcal{A}_{0}\Lambda_{0} + \cdots + \mathcal{A}_{l}\Lambda_{l}$ is the $\mathcal{P}$-canonical form (which we abbreviate by the $\mathcal{P}$-cf) of $A$ relative to $(\mathcal{S}^{\ast}, \mathcal{T})$.
    \item $\pmb{A}=N(A) + \mathcal{A'}_{0}\Gamma^{0} + \cdots + \mathcal{A'}_{l}\Gamma^{l}$ is the $\mathcal{P}$-cf of $A$ relative to $(\mathcal{S}^{\ast}, \mathcal{H})$.
    \end{itemize}
\section{Explicit expression of the minimal polynomial of the Kronecker product of matrices}
%
In this section we obtain the minimal polynomial of the Kronecker product of the companion matrices of any finite family of monic polynomials $P_{i}$, $1\leq i\leq m$ over a field $F$. As a consequence we obtain the monic polynomial $P$ such that $\prod\limits_{i=1}^{m}L(P_{i})=L(P)$, where $L(P)$ denote the vector space of all sequences over $F$ generated by the linear recurrence corresponding to the polynomial $P$.
\\The following definition is from~\cite{Gott} and it is equivalent to ours proposed in~\cite{Mouc2}.
\begin{definition}
For two positive integers $s$ and $t$, let $s\wedge t$ be the maximum value of $i+j+1$ such that $\binom{i+j}{i}\neq 0$ (in $F$) where $0\leq i\leq s-1$ and $0\leq j\leq t-1$.
\end{definition}
Let us recall two lemmas and a theorem proved in~\cite{Mouc2} that will be used in the proof of Theorem~\ref{Thm,,1}.
\begin{lem} Let $\lambda\in F$ and let $\wedge_{\lambda}$ be the noncommutative binary operation on $\mathbb{N}$ defined by $t\wedge_{\lambda} s=\min(t, s)\delta_{0,\lambda}+t\delta_{0,s}^{c}\delta_{0,\lambda}^{c}$ where $\delta_{e,f}^{c}=1-\delta_{e,f}$ and $\delta_{e,f}$ is the Kronecker symbol. Then For all $s,t\in \mathbb{N}$ and all $\lambda\in F$ we have $\langle0\rangle_{t}\langle\lambda\rangle_{s}=\langle0\rangle_{t\wedge_{\lambda} s}$.
\end{lem}
\begin{thrm}\label{thm 2221} Let $i,j\in \mathbb{N}$, $\lambda, \mu\in F^{\ast}=F-\{0\}$ and $\eta \in F$. Put $0\wedge s=s\wedge 0=0$ for all $s\in \mathbb{N}$. Then
\begin{enumerate}[(i)]
\item $\dim(\langle\eta\rangle_{i})=i$.
\item $\langle1\rangle_{i}\langle1\rangle_{j}=\langle1\rangle_{i\wedge j}$.
\item $\langle\lambda\rangle_{i}\langle\mu\rangle_{j}=\langle\lambda\mu\rangle_{i\wedge j}$.

\end{enumerate}
\end{thrm}
\begin{lem}\label{lem 197} Let $i,j$ be positive integers and $s,s',t,t'$ non-negative integers.
\begin{enumerate}[(1)]
\item If $F$ is a field with characteristic $0$, then we have $i\wedge j=i+j-1$.
\item If $s'\leq s$ and $t'\leq t$, then $s'\wedge t'\leq s\wedge t$ and $s'\wedge_{\lambda} t'\leq s\wedge_{\lambda} t$ for all $\lambda\in F$.

\end{enumerate}
\end{lem}
The following lemma is proved in~\cite{Mouc1}.
\begin{lem}\label{lem 325}
Let $A$ be a square matrix over $F$. Then the minimal polynomial of $A$ is $\mathcal{M}_{A}(X)=X^{t_{0}}\prod_{j=1}^{p}(X-\lambda_{j})^{t_{j}}$, where $t_{0}$ and $t_{j}, j\neq 0$, are respectively the greatest integers such that $\pmb{0}_{t_{0}-1}$ and $\pmb{\lambda}_{j}\Lambda_{t_{j}-1}$ appear in the $\mathcal{P}$-cf of $A$.
\end{lem}
\begin{thrm}\label{Thm,,1} Let $\alpha$ and $\beta$ be non zero elements of $F$. Let $A=\oplus_{i=1}^{m}J_{s_{i}}(\alpha)$, $B=\oplus_{i=1}^{n}J_{t_{i}}(\beta)$ and $C=\oplus_{i=1}^{m}J_{s_{i}}(0)$. Then
\begin{enumerate}[(1)]
\item The minimal polynomial of the matrix $J_{s}(1)\otimes J_{t}(1)$ is $(X-1)^{s\wedge t}$.
\item The minimal polynomial of the matrix $J_{s}(\alpha)\otimes J_{t}(\beta)$ is $(X-\alpha\beta)^{s\wedge t}$.
\item The minimal polynomial of the nilpotent matrix $J_{s}(0)\otimes J_{t}(\lambda)$ is $X^{s\wedge_{\lambda} t}$.
\item The minimal polynomial of the matrix $A\otimes B$ is $(X-\alpha\beta)^{\iota_{A}(\alpha)\wedge \iota_{B}(\beta)}$, where $\iota_{A}(\alpha)$ and $\iota_{B}(\beta)$ are the indexes of $\alpha$ and $\beta$ as eigenvalues of $A$ and $B$, respectively.
\item The minimal polynomial of the matrix $C\otimes A$ is $X^{\iota_{C}(0)\wedge_{\alpha} \iota_{A}(\alpha)}$.
    \end{enumerate}
\end{thrm}
\proof~
\begin{enumerate}[(1)]
\item Observe first that the geometric part of $J_{s}(1)\otimes J_{t}(1)$ is $(J_{s}(1)^{k}\otimes J_{t}(1)^{k})_{k\geq 0}$. Moreover, it is well-known that $$(J_{s}(1)^{k})_{k\geq 0}=[\Lambda_{0},\ldots,\Lambda_{s-1}]_{s}$$ (see, e.g, formula~(18) of~\cite{Mouc0}). Then $$(J_{s}(1)^{k}\otimes J_{t}(1)^{k})_{k\geq 0}=[\Lambda_{0},\ldots,\Lambda_{s-1}]_{s}\otimes [\Lambda_{0},\ldots,\Lambda_{t-1}]_{t}$$
Since all elements of this last matrix lie in its first row which is $$[\Lambda_{0}\Lambda_{0},\cdots,\Lambda_{0}\Lambda_{t-1},\cdots,\Lambda_{s-1}\Lambda_{0},\cdots,\Lambda_{s-1}\Lambda_{t-1}],$$ the greatest integer $m$ such that $\Lambda_{m-1}$ appear in the geometric part of $J_{s}(1)\otimes J_{t}(1)$ is exactly the dimension of the vector space $\langle1\rangle_{s}\langle1\rangle_{t}$. Therefore in view of Lemma~\ref{lem 325} we deduce that the minimal polynomial of the matrix $J_{s}(1)\otimes J_{t}(1)$ is $(X-1)^{s\wedge t}$.
\item The proof is similar to that of $(2)$
\item Since $J_{s}(1)\sim [1,\lambda,0\ldots,0]_{s}$ for all $\lambda\in F$ such that $\lambda\neq 0$, it follows that, for all $s,t\in \mathbb{N}^{\ast}$ we have $J_{s}(\alpha)\otimes J_{t}(\beta)\sim \alpha\beta J_{s}(1)\otimes J_{t}(1)$, and then the minimal polynomial of $J_{s}(\alpha)\otimes J_{t}(\beta)$ is $(X-\alpha\beta)^{s\wedge t}$.
    \item Since $A\otimes B\sim \oplus_{i,j}(J_{s_{i}}(\alpha)\otimes J_{t_{j}}(\beta))$, it follows that the minimal polynomial of the matrix $A\otimes B$ is $(X-\alpha\beta)^{\max_{i,j}(s_{i}\wedge t_{j})}=(X-\alpha\beta)^{\iota_{A}(\alpha)\wedge \iota_{B}(\beta)}$, in view of the second assertion of Lemma~\ref{lem 197}.
        \item The proof is similar to that of $(4)$
 \end{enumerate}
\endproof
To deal with the general case of any finite number of arbitrary square matrices, we need the following notations used in~\cite{Mouc2}.
\\Let $m$ be a positive integer and let $\mathcal{R}$ be the equivalence relation on ${\overline{F}^{\star}}^{m}$, $\overline{F}$ designates the algebraic closure of the field $F$, defined by
$$(\lambda_{1},\ldots,\lambda_{m})\mathcal{R}(\mu_{1},\ldots,\mu_{m}) \quad\text{if and only if}\quad \lambda_{1}\cdots\lambda_{m}=\mu_{1}\cdots\mu_{m}.$$
Let us partition ${\overline{F}^{\star}}^{m}$ into equivalence classes under the equivalence relation $\mathcal{R}$
\begin{equation*}
{\overline{F}^{\star}}^{m}=\bigcup_{i\in T} \Omega_{i}
\end{equation*}
Let $\wedge_{m}$ denote the map defined as follows
\begin{align*}
\wedge_{m}: \mathbb{N}^{m}&\longrightarrow \mathbb{N} \\
(t_{1},\ldots,t_{m})&\longmapsto t_{1}\wedge\cdots\wedge t_{m}
\end{align*}
For any arbitrary monic polynomials $P_{1},\ldots,P_{m}$, let us define the map $f=f_{P_{1},\ldots,P_{m}}$ as follows
\begin{align*}
f: {\overline{F}^{\star}}^{m}&\longrightarrow \mathbb{N}^{m}\\
(\lambda_{1},\ldots,\lambda_{m})&\longmapsto (\lambda_{1}(P_{1}),\ldots,\lambda_{m}(P_{m})).
\end{align*}
where $\lambda(P_{i})$ designates the multiplicity of $\lambda$ in $P_{i}$
\\let us use the following notations for simplicity
\begin{eqnarray*}
[\mu_{1}]_{f}\cdots[\mu_{m}]_{f}&=&[\mu_{1}]_{\mu_{1}(P_{1})}\cdots[\mu_{m}]_{\mu_{m}(P_{m})}\\
&=&[\mu_{1}\cdots\mu_{m}]_{\overline{f}(\mu_{1},\ldots,\mu_{m})}\\
\overline{f}&=&\wedge_{m}\circ f\\
\widehat{\Omega}_{i}&=&\mu_{1}\cdots\mu_{m}\quad\text{if}\quad (\mu_{1},\ldots,\mu_{m})\in \Omega_{i}\\
\Omega_{i}^{f}&=&\max_{(\mu_{1},\ldots,\mu_{m})\in \Omega_{i}}(\overline{f}(\mu_{1},\ldots,\mu_{m}))\\
\Upsilon(P_{1},\ldots,P_{m})&=&\prod_{i\in T}(X-\widehat{\Omega}_{i})^{\Omega_{i}^{f}}
\end{eqnarray*}
Under the above notations and assumptions, we have the following theorem:
\begin{thrm}\label{Thm 2453} Let $A_{i}, 1\leq i\leq m$ be square matrices over $F$ with minimal polynomials $\mathcal{M}_{A_{i}}(X)=X^{\iota_{A_{i}}(0)}Q_{i}$, and put $\Theta=\{i/Q_{i}= 1\}$. Then the minimal polynomial of the matrix $\otimes_{i=1}^{m}A_{i}$ is $X^{\rho}\Upsilon(Q_{1},\ldots,Q_{m}),$
where
\begin{eqnarray*}
\rho=\begin{cases}
\min\{\iota_{A_{i}}(0)/i\in \Theta\}&\quad\text{if}\quad \Theta\neq\emptyset\\
\max\{\iota_{A_{i}}(0)/1\leq i\leq m\}& \quad\text{otherwise}.
\end{cases}
\end{eqnarray*}
\end{thrm}
\begin{proof} There is no loss of generality to consider $m=2$, as the general case follows by induction. Let
\begin{equation*}
J(0)_{1}\oplus J(\alpha_{1})\oplus \cdots\oplus J(\alpha_{n})
 \end{equation*}
and
\begin{equation*}
J(0)_{2}\oplus J(\beta_{1})\oplus\cdots\oplus J(\beta_{q})
\end{equation*}
be respectively the Jordan canonical forms of $A_{1}$ and $A_{2}$ in $\overline{F}$. Since
\begin{align*}
A\otimes B\sim & [J(0)_{1}\otimes J(0)_{2}]\oplus [\oplus_{j=1}^{q}(J(0)_{1}\otimes J(\beta_{j}))]\oplus\\
&[\oplus_{i=1}^{n}(J(\alpha_{i})\otimes J(0)_{2})]\oplus [\oplus_{i,j\geq 1}(J(\alpha_{i})\otimes J(\beta_{j}))],
\end{align*}
it follows, in view of Theorem~\ref{Thm,,1}, that the minimal polynomial of the matrix $A\otimes B$ is $X^{s}Q(X)$ where
\begin{eqnarray*}
s&=&\max\{\iota_{A_{1}}(0)\wedge_{0} \iota_{A_{2}}(0), \iota_{A_{1}}(0)\wedge_{\beta_{j}} \iota_{A_{2}}(\beta_{j}), \iota_{A_{2}}(0)\wedge_{\alpha_{i}} \iota_{A_{1}}(\alpha_{i})/\\
&&1\leq i\leq n \,\text{and}\, 1\leq j\leq q\}\\
&=&\rho
\end{eqnarray*}
and $Q$ is the least common multiple of the polynomials
\begin{equation*}
(X-\alpha_{i}\beta_{j})^{\iota_{A1}(\alpha_{i})\wedge \iota_{A_{2}}(\beta_{j})},\,\, 1\leq i\leq n \,\text{and}\, 1\leq j\leq q
\end{equation*}
which is exactly the polynomial $\Upsilon(Q_{1},Q_{2})$ by definition of this last polynomial. The proof is finished.
\end{proof}
As an immediate consequence of Theorem~$2.13$ of~\cite{Mouc2}, Theorem~\ref{Thm 2453} above and the fact that every monic polynomial is the minimal polynomial of its companion matrix we have the following result
\begin{thrm}
Let $P_{i}, 1\leq i\leq m$ be any monic polynomials over $F$. Then we have
$$\prod_{i=1}^{m}L(P_{i})=L(P)$$
where $P(X)\in F[X]$ is the minimal polynomial of the Kronecker product of the companion matrices of $P_{i}, 1\leq i\leq m$.
\end{thrm}
\begin{remark} In their paper \cite{Cerr} U. Cerruti and F. Vaccarino proved that if $P_{i}, 1\leq i\leq m$ are monic polynomials over a commutative ring with identity $R$, then $\prod_{i=1}^{m}L(P_{i})\subseteq L(H)$
where $H(X)\in R[X]$ is the characteristic polynomial of the Kronecker product of the companion matrices of $P_{i}, 1\leq i\leq m$.
\end{remark}
\section{$\mathcal{P}$-canonical form of the exponential of matrices}
The following result shows how to deduce the $\mathcal{P}$-cf relative to $(\mathcal{S}^{\ast}, \mathcal{H})$ of the matrix exponential function $e^{tA}$, $A$ is a square matrix with complex elements, from the $\mathcal{P}$-cf of $A$ relative to $(\mathcal{S}^{\ast}, \mathcal{T})$.
\begin{thrm}\label{thm,30}
Let $A\in M_{q}(\mathbb{C})$ and let
\begin{equation}\label{eqq 1}
\pmb{A}= N(A) + \mathcal{A}_{0}\Lambda_{0} + \cdots + \mathcal{A}_{m}\Lambda_{m}
\end{equation}
be the $\mathcal{P}$-cf relative to $(\mathcal{S}^{\ast}, \mathcal{T})$ of $A$. Then for every $t\in \mathbb{C}$, the $\mathcal{P}$-cf relative to $(\mathcal{S}^{\ast}, \mathcal{H})$ of $\e^{tA}$ can be easily obtained by transforming the $\mathcal{P}$-cf relative to $(\mathcal{S}^{\ast}, \mathcal{T})$ of $A$ using the following substitutions:
\begin{align*}
\pmb{\lambda}\Lambda_{i}&\hookrightarrow \frac{(t\lambda)^{i}\pmb{\e^{t\lambda}}\Gamma^{i}}{i!}\\
\pmb{0}_{i}&\hookrightarrow \frac{t^{i}\Gamma^{i}}{i!}
\end{align*}
\end{thrm}
\begin{proof}
Let $\mathcal{M}_{A}(X)=X^{t_{0}}\prod_{j=1}^{p}(X-\lambda_{j})^{t_{j}}$ (possibly $t_{0}=0$), be the minimal polynomial of $A$ and let $t\in \mathbb{C}$. Let $\pi_{0},\ldots,\pi_{p}$ be the spectral projections of $A$ at $0,\lambda_{1},\ldots,\lambda_{p}$ respectively. From Proposition~$3.1$ of~\cite{Mouc1}, we have
\begin{equation*}
\pmb{A}=\sum_{i=0}^{t_{0}-1}\pmb{0}_{i}A^{i}\pi_{0}+
\sum_{i=0}^{m}\sum_{j=1}^{p}\pmb{\lambda}_{j}\lambda_{j}^{-i}(A-\lambda_{j}I_{q})^{i}\pi_{j}\Lambda_{i}
\end{equation*}
Then
\begin{equation*}
\begin{split}
\pmb{tA}&=\sum_{i=0}^{t_{0}-1}\pmb{0}_{i}\pmb{t}A^{i}\pi_{0}+
\sum_{i=0}^{m}\sum_{j=1}^{p}\pmb{t\lambda}_{j}\lambda_{j}^{-i}(A-\lambda_{j}I_{q})^{i}\pi_{j}\Lambda_{i}\\
&=\sum_{i=0}^{t_{0}-1}\pmb{0}_{i}t^{i}A^{i}\pi_{0}+
\sum_{i=0}^{m}\sum_{j=1}^{p}\pmb{t\lambda}_{j}(t\lambda_{j})^{-i}t^{i}(A-\lambda_{j}I_{q})^{i}\pi_{j}\Lambda_{i},
\end{split}
\end{equation*}
where $m=\max\{t_{1}-1,\ldots,t_{p}-1\}$. Dividing by $k!$, we obtain
$$\frac{(tA)^{k}}{k!}=\sum_{i=0}^{t_{0}-1}\delta_{i,k}\frac{t^{i}}{k!}A^{i}\pi_{0}+
\sum_{i=0}^{m}\sum_{j=1}^{p}\frac{1}{k!}(t\lambda_{j})^{k-i}\binom{k}{i}t^{i}(A-\lambda_{j}I_{q})^{i}\pi_{j}.$$
Therefore
\begin{equation*}
\begin{split}
\e^{tA}&=\sum_{i=0}^{t_{0}-1}\sum_{k=0}^{\infty}\delta_{i,k}\frac{t^{i}}{k!}A^{i}\pi_{0}+
\sum_{i=0}^{m}\sum_{j=1}^{p}\sum_{k=i}^{\infty}\frac{1}{k!}(t\lambda_{j})^{k-i}\binom{k}{i}t^{i}(A-\lambda_{j}I_{q})^{i}\pi_{j}\\
&=\sum_{i=0}^{t_{0}-1}\frac{t^{i}}{i!}A^{i}\pi_{0}+
\sum_{i=0}^{m}\sum_{j=1}^{p}\sum_{k=i}^{\infty}\frac{t^{i}}{i!}\frac{(t\lambda_{j})^{k-i}}{(k-i)!}(A-\lambda_{j}I_{q})^{i}\pi_{j}\\
&=\sum_{i=0}^{t_{0}-1}\frac{t^{i}}{i!}A^{i}\pi_{0}+
\sum_{i=0}^{m}\sum_{j=1}^{p}\frac{t^{i}}{i!}\e^{t\lambda_{j}}(A-\lambda_{j}I_{q})^{i}\pi_{j}
\end{split}
\end{equation*}
which is a well known result. But since $(\e^{tA})^{k}=\e^{ktA}$ for all $k\geq 0$, it follows that
\begin{equation}\label{eq 001}
\begin{split}
\pmb{\e^{tA}}&=\sum_{i=0}^{t_{0}-1}\frac{t^{i}\Gamma^{i}}{i!}A^{i}\pi_{0}+
\sum_{j=1}^{p}\pmb{\e^{t\lambda_{j}}}(A-\lambda_{0}I_{q})^{0}\pi_{j}
\frac{t^{0}\Gamma^{0}}{0!}+
\cdots\\
&+\sum_{j=1}^{p}\pmb{\e^{t\lambda_{j}}}(A-\lambda_{j}I_{q})^{m}\pi_{j}\frac{t^{m}\Gamma^{m}}{m!}.
\end{split}
\end{equation}
So
\begin{equation*}
\begin{split}
\pmb{\e^{tA}}&=\sum_{j=1}^{p}(\pmb{\e^{t\lambda_{j}}}(A-\lambda_{0}I_{q})^{0}\pi_{j}+A^{0}\pi_{0})
\frac{t^{0}}{0!}\Gamma^{0}+\cdots\\
&+\sum_{j=1}^{p}(\pmb{\e^{t\lambda_{j}}}(A-\lambda_{j}I_{q})^{d}\pi_{j}+A^{d}\pi_{0})\frac{t^{d}}{d!}\Gamma^{d},\end{split}
\end{equation*}
where $d=\max\{m,t_{0}-1\}$, which is the $\mathcal{P}$-cf relative to $(\mathcal{S}^{\ast}, \mathcal{H})$ of $\e^{tA}$.
To conclude, all that remains is to compare equation~\ref{eq 001} with the $\mathcal{P}$-cf of $A$ relative to $(\mathcal{S}^{\ast}, \mathcal{T})$.
\end{proof}
The following Theorem is an interesting consequence of Theorem~\ref{thm,30}.
\begin{thrm} Let $A$ be a nonsingular complex matrix such that $\chi_{A}\in \mathbb{R}[X]$. Let
\begin{equation*}
\pmb{A}=N(A) + \mathcal{A}_{0}\Lambda_{0} + \cdots + \mathcal{A}_{m}\Lambda_{m}
\end{equation*}
be the $\mathcal{P}$-cf relative to $(\mathcal{S}^{\ast}, \mathcal{T})$ of $A$.
Suppose that the sequences appearing in $\mathcal{A}_{0},\ldots,\mathcal{A}_{m}$ are
\begin{equation*}
r_{i}^{k}\cos(k\theta) \quad(\text{or}\quad r_{i}^{k}\sin(k\theta)), 1\leq i\leq s, \quad\text{and}\quad \pmb{\lambda}_{i}, 1\leq i\leq l
\end{equation*}
Then we have
\begin{enumerate}[(1)]
\item $\mu_{1}=r_{1}\e^{\mathrm{i}\theta_{1}},\ldots,\mu_{s}=r_{1}\e^{\mathrm{i}\theta_{s}},\overline{\mu}_{1}=r_{1}\e^{-\mathrm{i}\theta_{1}},
\ldots,\overline{\mu}_{s}=r_{s}\e^{-\mathrm{i}\theta_{s}},\lambda_{1},\ldots,\lambda_{l}$ are the eigenvalues, not necessarily distinct, of $A$.
\item The $\mathcal{P}$-cf relative to $(\mathcal{S}^{\ast}, \mathcal{H})$ of the matrix $\e^{tA}$ can be easily obtained by transforming the $\mathcal{P}$-cf relative to $(\mathcal{S}^{\ast}, \mathcal{T})$ of $A$ using the following substitutions:
\begin{align*}
\Lambda_{j}(r_{i}^{k}\cos(k\theta_{i}))_{k}&\hookrightarrow
\operatorname{Re}(f_{j}(tk\mu_{i}))_{k}&\\
\Lambda_{j}(r_{i}^{k}\sin(k\theta_{i}))_{k}&\hookrightarrow \operatorname{Im}(f_{j}(tk\mu_{i}))_{k}&\\
\pmb{\lambda_{i}}\Lambda_{j}&\hookrightarrow \frac{(t\lambda_{i})^{j}\pmb{\e^{t\lambda_{i}}}\Gamma^{j}}{j!},& 1\leq i\leq l\\
\pmb{0}_{j}&\hookrightarrow \displaystyle\frac{t^{j}\Gamma^{j}}{j!}&
\end{align*}
where for all $z\in \mathbb{C}$,
\begin{eqnarray*}
f_{j}(z)=\begin{cases}
\frac{z^{j}}{j!}\e^{z} & \text{if}\quad j\neq 0\\
\e^{z} &\text{if}\quad j=0
\end{cases}
\end{eqnarray*}
\end{enumerate}
\end{thrm}
\begin{proof}
The proof is a straightforward application of Theorem~\ref{thm,30}, and the fact that if $\mu=r\e^{\mathrm{i}\theta}$ then
$$r^{k}\cos(k\theta)=\frac{\mu^{k}+\overline{\mu}^{k}}{2}$$
and
$$r^{k}\sin(k\theta)=\frac{\mu^{k}-\overline{\mu}^{k}}{2\mathrm{i}}.$$
\end{proof}
\section{$\mathcal{P}$-canonical form of the logarithm of matrices}
It is well known that a complex matrix $A$ has a logarithm if and only if $A$ is nonsingular. The following theorem shows that knowing the $\mathcal{P}$-cf relative to $(\mathcal{S}^{\ast}, \mathcal{T})$ (or $(\mathcal{S}^{\ast}, \mathcal{H})$) of a nonsingular matrix $A$, we can derive a logarithm of it.
\begin{thrm}\label{thm,22} Let $A$ be a nonsingular matrix of order $q$ over $\mathbb{C}$ and let
\begin{equation*}
\lambda_{1}=\e^{z_{1}},\ldots,\lambda_{s}=\e^{z_{s}}
\end{equation*}
its eigenvalues. Let us denote by $(A(k))_{k\geq0}$ the $\mathcal{P}$-cf relative to $(\mathcal{S}^{\ast}, \mathcal{T})$ of $A$, and consider the matrix complex-valued smooth function $A(t)$ of a real variable $t$ obtained by plugging the variable $t$ for $k$ in $A(k)$. Then we have
\begin{enumerate}[(1)]
\item The derivative $A'(0)$ of $A(t)$ at $t=0$ is a matrix logarithm of $A$.
\item The $\mathcal{P}$-cf relative to $(\mathcal{S}^{\ast}, \mathcal{T})$ of $A'(0)$ can be easily obtained by transforming the $\mathcal{P}$-cf relative to $(\mathcal{S}^{\ast}, \mathcal{H})$ of $A$ using the following substitutions:
$$\begin{array}{ccccc}
\pmb{\lambda_{j}}\Gamma^{i}&\hookrightarrow & i!\pmb{z_{j}}z_{j}^{-i}\Lambda_{i}&\text{if}& z_{j}\neq 0\\
\pmb{\lambda_{j}}\Gamma^{i}&\hookrightarrow & i!\pmb{0}_{i}&\text{if}& z_{j}=0
\end{array}$$
\item  The eigenvalues of $A'(0)$ are $z_{1},\ldots,z_{s}$. More precisely, if $\chi_{A}=(X-\lambda_{1})^{m_{1}} \,\cdots\, (X-\lambda_{s})^{m_{s}}$ is the characteristic polynomial of $A$, then $A'(0)$ is the unique logarithm of $A$ with characteristic polynomial $\chi_{A'(0)}=(X-z_{1})^{m_{1}} \,\cdots\, (X-z_{s})^{m_{s}}$.
\item If $A$ has nonnegative real eigenvalues, then the principal logarithm of $A$ is equal to the matrix $A'(0)$ obtained when we choose the $z_{i}$ to be in the strip $\{z\in \mathbb{C}/ -\pi< Im(z)< \pi\}$.
\end{enumerate}
\end{thrm}
\begin{proof}~
\begin{enumerate}[(1)]
\item Let $\binom{t}{i}=\displaystyle\frac{t(t-1)\cdots(t-i+1)}{i!}$ be the binomial polynomial of degree $i$. Since, for $i\geq 1$, the coefficient of $t$ in that polynomial is $\displaystyle\frac{(-1)^{i-1}}{i}$, it follows that the derivative of $\binom{t}{i}$ at $0$ is $\binom{t}{i}'(0)=\displaystyle\frac{(-1)^{i-1}}{i}$ for all integer $i\geq 1$. On the other hand, we have
\begin{equation*}
\mathcal{A}_{0}'(0)=z_{1}\pi_{1}+\cdots+z_{s}\pi_{s}.
\end{equation*}
Hence
\begin{eqnarray*}
A'(0)&=&z_{1}\pi_{1}+\cdots+z_{s}\pi_{s}+\sum_{i=1}^{m} \frac{(-1)^{i-1}}{i}\mathcal{A}_{i}(0)\\
&=&z_{1}\pi_{1}+\cdots+z_{s}\pi_{s}+\sum_{i=1}^{m} \frac{(-1)^{i-1}}{i}\sum_{j=1}^{s}\lambda_{j}^{-i}(A-\lambda_{j}I_{q})^{i}\pi_{j}.
\end{eqnarray*}
Since the matrix $(\lambda_{j}^{-1}A-I_{q})\pi_{j})^{m+1}=0$, it follows that
\begin{equation*}
C_{j}=\displaystyle\sum_{i=1}^{m} \frac{(-1)^{i-1}}{i}(\lambda_{j}^{-1}A-I_{q})^{i}\pi_{j}
\end{equation*}
is a logarithm of $I_{q}+(\lambda_{j}^{-1}A-I_{q})\pi_{j}$. Consequently, as the $C_{j}$s are pairwise commuting matrices, we have $C=\displaystyle\sum_{j=1}^{s}C_{j}$ is a logarithm of
\begin{eqnarray*}
\prod_{j=1}^{s}(I_{q}+(\lambda_{j}^{-1}A-I_{q})\pi_{j})&=&I_{q}+\sum_{j=1}^{s}(\lambda_{j}^{-1}A-I_{q})\pi_{j}\\
&=&A(\sum_{j=1}^{s}\lambda_{j}^{-1}\pi_{j})\\
&=&A\e^{-(z_{1}\pi_{1}+\cdots+z_{s}\pi_{s})}.
\end{eqnarray*}
Hence
\begin{equation*}
A=\e^{C}\e^{-(z_{1}\pi_{1}+\cdots+z_{s}\pi_{s})}.
\end{equation*}
Finally, since the matrix $C$ commutes with $z_{1}\pi_{1}+\cdots+z_{s}\pi_{s}$, it follows that $A'(0)$ is a logarithm of $A$.
\item It is sufficient to use the argument of Proposition~\ref{thm,30} in reverse direction.
\item Follows from property $(2)$ above and Corollary~(3.6) of~\cite{Mouc1} which allows us to deduce the minimal polynomial, and in particular the eigenvalues, of a matrix from its $\mathcal{P}$-cf
\item Follows directly from property $(3)$ above.
\end{enumerate}
\end{proof}
As a consequence of Theorem~\ref{thm,22} we have the following result.
\begin{cor}\label{corl,1} Let $A$ be a nonsingular matrix of order $q$ over $\mathbb{C}$
and let $\lambda_{1}=\e^{z_{1}},\ldots,\lambda_{s}=\e^{z_{s}}$ its eigenvalues. Let $P_{1,ij}(X),\ldots,P_{s,ij}(X)\in \mathbb{C}[X]$ such that
$$(A^{k})_{ij}=P_{1,ij}(k)\lambda_{1}^{k}+\ldots+P_{s,ij}(k)\lambda_{s}^{k}$$
for all $k\geq 0$. Here, $(A^{k})_{ij}$ denotes the (i,j)-th entry of the matrix $A^{k}$. Assume
$$\begin{array}{ccc}
P_{1,ij}(X)&=&p_{1,ij,0}+p_{1,ij,1}X+\cdots\\
\vdots&&\vdots\\
P_{s,ij}(X)&=&p_{s,ij,0}+p_{s,ij,1}X+\cdots
\end{array}$$
Then the matrix
$$(\sum_{e=1}^{s}p_{e,ij,0}+\sum_{e=1}^{s}p_{e,ij,1}z_{e})_{ij}$$
is a logarithm of $A$.
\end{cor}
\proof
The proof Follows immediately from property $(1)$ of Theorem~\ref{thm,22}.
\endproof
The following theorem is also a consequence of Theorem~\ref{thm,22}.
\begin{thrm}\label{thm, 2321} Let $A$ be a nonsingular complex matrix such that $\chi_{A}\in \mathbb{R}[X]$. Let
\begin{equation*}
\pmb{A}=N(A) + \mathcal{A}_{0}\Gamma^{0} + \cdots + \mathcal{A}_{l}\Gamma^{l}
\end{equation*}
be the $\mathcal{P}$-cf relative to $(\mathcal{S}^{\ast}, \mathcal{H})$ of $A$.
Suppose that the sequences appearing in $\mathcal{A}_{0},\ldots,\mathcal{A}_{l}$ are
\begin{equation*}
r_{i}^{k}\cos(k\theta) \quad(\text{or}\quad r_{i}^{k}\sin(k\theta)), 1\leq i\leq s, \quad\text{and}\quad \pmb{\lambda}_{i}, 1\leq i\leq m
\end{equation*}
Then we have
\begin{enumerate}[(1)]
\item $\mu_{1}=r_{1}\e^{\mathrm{i}\theta_{1}},\ldots,\mu_{s}=r_{1}\e^{\mathrm{i}\theta_{s}},\overline{\mu}_{1}=r_{1}\e^{-\mathrm{i}\theta_{1}},
\ldots,\overline{\mu}_{s}=r_{s}\e^{-\mathrm{i}\theta_{s}},\lambda_{1},\ldots,\lambda_{m}$ are the eigenvalues, not necessarily distinct, of $A$.
\item Let $z_{1},\ldots,z_{m}$ be logarithms of $\lambda_{1},\ldots,\lambda_{m}$, $w_{1},\ldots,w_{s}$ be logarithms of $\mu_{1},\ldots,\mu_{s}$ and $u_{1},\ldots,u_{s}$ be logarithms of $\overline{\mu}_{1},\ldots,\overline{\mu}_{s}$. Then the $\mathcal{P}$-cf relative to $(\mathcal{S}^{\ast}, \mathcal{T})$ of the logarithm $A'(0)$ of $A$ can be easily obtained by transforming the
    $\mathcal{P}$-cf relative to $(\mathcal{S}^{\ast}, \mathcal{H})$ of $A$ using the following substitutions:
\renewcommand{\arraystretch}{1.5}
$$\begin{array}{ccccc}
\Gamma^{j}(r_{i}^{k}\cos(k\theta_{i}))_{k}&\hookrightarrow & \frac{1}{2}(g_{(j,k)}(w_{i})+g_{(j,k)}(u_{i}))\\
\Gamma^{j}(r_{i}^{k}\sin(k\theta_{i}))_{k}&\hookrightarrow & \frac{1}{2\mathrm{i}}(g_{(j,k)}(w_{i})-g_{(j,k)}(u_{i}))\\
\pmb{\lambda_{j}}\Gamma^{i}&\hookrightarrow & i!\pmb{z_{j}}z_{j}^{-i}\Lambda_{i}&\text{if}\quad z_{j}\neq 0\\
\pmb{\lambda_{j}}\Gamma^{i}&\hookrightarrow & i!\pmb{0}_{i}&\text{if}\quad z_{j}=0
\end{array}$$
where for all $z\in \mathbb{C}$,
\begin{eqnarray*}
g_{(j,k)}(z)=\begin{cases}
j!\Lambda_{j}\pmb{z}z^{-j} & \text{if}\quad z\neq 0\\
j!\pmb{0}_{j} &\text{if}\quad z=0
\end{cases}
\end{eqnarray*}

\item $z_{1},\ldots,z_{m},w_{1},\ldots,w_{s},u_{1},\ldots,u_{s}$ are the eigenvalues, not necessarily distinct, of $A'(0)$.
\end{enumerate}
\end{thrm}
\begin{proof}
The proof is a straightforward application of Theorem~\ref{thm,22}, and hence will be omitted.
\end{proof}
\section{Illustration examples}
In this section we give some theoretical and numerical illustrations of our methods.
\begin{example} We begin by the trivial and well known case of a nilpotent matrix. Suppose that $A^{q}=0$. Then we have $A^{k}=\sum_{i=0}^{q-1}\pmb{0}_{i}A^{i}$ and hence $\e^{kA}=\sum_{i=0}^{q-1}\frac{k^{i}}{i!}A^{i}$ in virtue of Theorem~\ref{thm,30}.
\end{example}
\begin{example}
Let \begin{equation*}
A=[a_{0},a_{1},\ldots,a_{n-1}]=
\begin{pmatrix}
a_{0}& a_{1}&. &.&a_{n-1}\\
0&a_{0}&a_{1}&.&a_{n-2}\\
.\;&\ddots &\ddots& \ddots& .\;\;\;\\
.\;& &\ddots& \ddots &a_{1}\\
0&. &. &0&a_{0}
\end{pmatrix}.
\end{equation*}
a semicirculant matrix. Using the method given in~\cite{Mouc0} together with the methods obtained in this paper, we can compute easily the exponential and the logarithm of $A$. For example let consider Example~$3.4$ of~\cite{Mouc0} $A=[2,4,2,3]$. We have $$A^{k}=[\pmb{a}_{0}(k), \pmb{a}_{1}(k), \pmb{a}_{2}(k), \pmb{a}_{3}(k)],$$ where
\begin{align*}
\pmb{a}_{0}(k)&=2^{k}\binom{k}{0}\\
\pmb{a}_{1}(k)&=2^{k+1}\binom{k}{1}\\
\pmb{a}_{2}(k)&=2^{k+2}\binom{k}{2}+2^{k}\binom{k}{1}\\
\pmb{a}_{3}(k)&=2^{k+3}\binom{k}{3}+2^{k+2}\binom{k}{2}+3\times2^{k-1}\binom{k}{1}
\end{align*}
Hence $\e^{kA}=[\pmb{b}_{0}(k),\pmb{b}_{1}(k),\pmb{b}_{2}(k),\pmb{b}_{3}(k)]$
where
\begin{align*}
\pmb{b}_{0}(k)&=\e^{2k}\\
\pmb{b}_{1}(k)&=4k\e^{2k}\\
\pmb{b}_{2}(k)&=(8k^{2}+2k)\e^{2k}\\
\pmb{a}_{3}(k)&=(\frac{2^{5}k^{3}}{3}+8k^{2}+3k)\e^{2k}
\end{align*}
On the other hand, we have
\begin{align*}
\pmb{a}_{0}(k)&=2^{k}\\
\pmb{a}_{1}(k)&=2^{k+1}k\\
\pmb{a}_{2}(k)&=2^{k+1}k^{2}-2^{k}k\\
\pmb{a}_{3}(k)&=\frac{2^{k+2}}{3}k^{3}-2^{k+1}k^{2}+\frac{13}{6}2^{k}k
\end{align*}
Then $A'(0)^{k}=[\pmb{c}_{0}(k),\pmb{c}_{1}(k),\pmb{c}_{2}(k),\pmb{c}_{3}(k)]$
where
\begin{align*}
\pmb{c}_{0}(k)&=\ln(2)^{k}\\
\pmb{c}_{1}(k)&=2\ln(2)^{k-1}\binom{k}{1}\\
\pmb{c}_{2}(k)&=2\ln(2)^{k-2}\binom{k}{2}-\ln(2)^{k-1}\binom{k}{1}\\
\pmb{c}_{3}(k)&=\frac{4}{3}\ln(2)^{k-3}\binom{k}{3}-2\ln(2)^{k-2}\binom{k}{2}+\frac{13}{6}\ln(2)^{k-1}\binom{k}{1}
\end{align*}
$A'(0)$ is the logarithm of $A$ with characteristic polynomial $(X-\ln(2))^{3}$.
Let $J_{n+1}(\lambda)=[\lambda,1,0,\ldots,0]$ the Jordan block of order $n+1$ with eigenvalue $\lambda\neq 0$. It is well-known that for all nonnegative integer $k$
$$J_{n+1}(\lambda)^{k}=[\lambda^{k},\lambda^{k-1}\binom{k}{1},\ldots,\lambda^{k-i}\binom{k}{i},\ldots,\lambda^{k-n}\binom{k}{n}]$$
Hence we get the following well-known result
$$\e^{kJ_{n+1}(\lambda)}=[\e^{k\lambda},k\e^{k\lambda},\ldots, \frac{\e^{k\lambda}k^{i}}{i!},\ldots,\frac{\e^{k\lambda}k^{n}}{n!}]$$
on the other hand, Since $$\binom{k}{i}=\sum_{m=0}^{i}\frac{s(i,m)}{i!}k^{m}$$ where $s(m,k)$ are the Stirling numbers of the first kind (see, e.g. Quaintance and Gould~\cite{Qua}), it follows that
$$\lambda^{k-i}\binom{k}{i}=\sum_{m=0}^{i}\frac{s(i,m)}{i!}\lambda^{-i}\lambda^{k}k^{m}$$
Hence
$\log(J_{n+1}(\lambda))^{k}=[d_{0}(k),\ldots,d_{n}(k)]$
\\where
$$
d_{i}(k)=\sum_{m=0}^{i}\frac{s(i,m)}{i!}m!\lambda^{-i}\ln(\lambda)^{k-m}\binom{k}{m}$$
In particular $$D=\ln(J_{n+1}(\lambda))=[\ln[\lambda),\lambda^{-1},\ldots,(-1)^{i-1}\frac{\lambda^{-i}}{i},\ldots,(-1)^{n-1}\frac{\lambda^{-n}}{n}]$$
where we have used the fact that $\frac{s(i,1)}{i!}=\frac{(-1)^{i-1}}{i}$ for all positive integer $i$. It is possible to verify that $J_{n+1}(\lambda)=\e^{D}$ through direct calculation, see e.g. Appendix F in~{Dick}. 
\end{example}
\begin{example}
Let $$A=\begin{pmatrix}
1&1&1&0\\1&1&1&-1\\0&0&-1&1\\0&0&1&-1
\end{pmatrix}$$
and
$$A(k)=\begin{pmatrix}
2^{-1+k}&2^{-1+k}&\frac{1}{16}2^{k}((-1)^{1+k}+5)&\frac{1}{16}2^{k}((-1)^{k}-1)\vspace*{0.1pc}\\
2^{-1+k}&2^{-1+k}&\frac{5}{16}2^{k}((-1)^{1+k}+1)&\frac{1}{16}2^{k}(5(-1)^{k}-1)\vspace*{0.1pc}\\
0&0&(-1)^{k}2^{-1+k}&(-1)^{1+k}2^{-1+k}\vspace*{0.1pc}\\
0&0&(-1)^{1+k}2^{-1+k}&(-1)^{k}2^{-1+k}
\end{pmatrix}.$$
Then we have
\begin{itemize}
\item The non-geometric part of $A$ is $(I_{4}-A(0))\pmb{0}_{0}+(A-A(1))\pmb{0}_{1}$.
\item The geometric part of $A$ is
$$\begin{pmatrix}
2^{-1}(2^{k})&2^{-1}(2^{k})&\frac{5}{16}(2^{k})-\frac{1}{16}((-2)^{k})&\frac{-1}{16}(2^{k})+\frac{1}{16}((-2)^{k})\vspace*{0.1pc}\\
2^{-1}(2^{k})&2^{-1}(2^{k})&\frac{5}{16}(2^{k})-\frac{5}{16}((-2)^{k})&\frac{-1}{16}(2^{k})+\frac{5}{16}((-2)^{k})\vspace*{0.1pc}\\
0&0&2^{-1}((-2)^{k})&-2^{-1}((-2)^{k})\vspace*{0.1pc}\\
0&0&-2^{-1}((-2)^{k})&2^{-1}((-2)^{k})
\end{pmatrix}$$
\item For all $k\in \mathbb{Z}$
\begin{eqnarray*}
\e^{ktA}&=&I_{4}-A(0)+(A-A(1))kt+\begin{pmatrix}
\frac{\e^{2kt}}{2}&\frac{\e^{2kt}}{2}&\frac{5\e^{2kt}-\e^{-2kt}}{16}&
\frac{-\e^{2kt}+\e^{-2kt}}{16}\vspace*{0.1pc}\\
\frac{\e^{2kt}}{2}&\frac{\e^{2kt}}{2}&\frac{5\e^{2kt}-5\e^{-2kt}}{16}&\frac{-\e^{2kt}+
5\e^{-2kt}}{16}\vspace*{0.1pc}\\
0&0&\frac{\e^{-2kt}}{2}&\frac{-\e^{-2kt}}{2}\vspace*{0.1pc}\\
0&0&\frac{-\e^{-2kt}}{2}&\frac{\e^{-2kt}}{2}
\end{pmatrix}
\end{eqnarray*}
Since \begin{eqnarray*}
(A-A(1))&=&\begin{pmatrix} 0&0&\frac{1}{4}&\frac{1}{4}\\
0&0&\frac{-1}{4}&\frac{-1}{4}\\
0&0&0&0\\0&0&0&0
\end{pmatrix}\\
I_{4}-A(0)&=&\begin{pmatrix} \frac{1}{2}&\frac{-1}{2}&\frac{-1}{4}&0\\
\frac{-1}{2}&\frac{1}{2}&0&\frac{-1}{4}\\
0&0&\frac{1}{2}&\frac{1}{2}\\0&0&\frac{1}{2}&\frac{1}{2}
\end{pmatrix}
\end{eqnarray*}
we have for all $k\in \mathbb{Z}$
\begin{equation*}
\e^{ktA}=\begin{pmatrix}
\frac{\e^{2kt}+1}{2}&\frac{\e^{2kt}-1}{2}&\frac{5\e^{2kt}-\e^{-2kt}+4kt-4}{16}&
\frac{-\e^{2kt}+\e^{-2kt}+4kt}{16}\vspace*{0.1pc}\\
\frac{\e^{2kt}-1}{2}&\frac{\e^{2kt}+1}{2}&\frac{5\e^{2kt}-5\e^{-2kt}-4kt}{16}&\frac{-\e^{2kt}+
5\e^{-2kt}-4kt-4}{16}\vspace*{0.1pc}\\
0&0&\frac{\e^{-2kt}+1}{2}&\frac{-\e^{-2kt}+1}{2}\vspace*{0.1pc}\\
0&0&\frac{-\e^{-2kt}+1}{2}&\frac{\e^{-2kt}+1}{2}
\end{pmatrix}
\end{equation*}
\end{itemize}
\end{example}
\begin{example}
Let
$$E=\begin{pmatrix}
2\sqrt{3}-10&2\sqrt{3}-23&\sqrt{3}-5\\4&\sqrt{3}+9&2\\
-2\sqrt{3}+2&-4\sqrt{3}+5&-\sqrt{3}+1
\end{pmatrix}$$
and put
$$E(k)=
\begin{pmatrix}
e_{11}(k)&e_{12}(k)&e_{13}(k)\\e_{21}(k)&e_{22}(k)&e_{23}(k)\\e_{31}(k)&e_{32}(k)&e_{33}(k)
\end{pmatrix}$$
where
\begin{eqnarray*}
e_{11}(k)&=&2^{k+1}(\cos(\frac{k\pi}{6})-5\sin(\frac{k\pi}{6}))\\
e_{12}(k)&=&2^{k+1}(\cos(\frac{k\pi}{6})-\frac{23}{2}\sin(\frac{k\pi}{6}))\\
e_{13}(k)&=&2^{k}(\cos(\frac{k\pi}{6})-5\sin(\frac{k\pi}{6}))\\
e_{21}(k)&=&2^{k+2}\sin(\frac{k\pi}{6})\\
e_{22}(k)&=&2^{k}(\cos(\frac{k\pi}{6})+9\sin(\frac{k\pi}{6}))\\
e_{23}(k)&=&2^{k+1}\sin(\frac{k\pi}{6})\\
e_{31}(k)&=&-2^{k+1}(\cos(\frac{k\pi}{6})-\sin(\frac{k\pi}{6}))
\end{eqnarray*}
\begin{eqnarray*}
e_{32}(k)&=&-2^{k+2}(\cos(\frac{k\pi}{6})-\frac{5}{4}\sin(\frac{k\pi}{6}))\\
e_{33}(k)&=&-2^{k}(\cos(\frac{k\pi}{6})-\sin(\frac{k\pi}{6}))
\end{eqnarray*}
We have, for all $k\geq1$,
$E^{k}=E(k)$.
Hence For all $k\in \mathbb{Z}$
\begin{eqnarray*}
\e^{ktE}&=&I_{3}-E(0)+
\begin{pmatrix}
f_{11}(k,t)&f_{12}(k,t)&f_{13}(k,t)\\f_{21}(k,t)&f_{22}(k,t)&f_{23}(k,t)\\f_{31}(k,t)&f_{32}(k,t)&f_{33}(k,t)
\end{pmatrix}\\
&=&\begin{pmatrix}
f_{11}(k,t)-1&f_{12}(k,t)-2&f_{13}(k,t)-1\\f_{21}(k,t)&f_{22}(k,t)&f_{23}(k,t)\\f_{31}(k,t)+2&f_{32}(k,t)+4&f_{33}(k,t)+2
\end{pmatrix}
\end{eqnarray*}
where
\begin{eqnarray*}
f_{11}(k,t)&=&(1+5\mathrm{i})\e^{2tk\e^{\frac{\pi\mathrm{i}}{6}}}+(1-5\mathrm{i})\e^{2tk\e^{\frac{-\pi\mathrm{i}}{6}}}\\
f_{12}(k,t)&=&(1+\frac{23}{2}\mathrm{i})\e^{2tk\e^{\frac{\pi\mathrm{i}}{6}}}+
(1-\frac{23}{2}\mathrm{i})\e^{2tk\e^{\frac{-\pi\mathrm{i}}{6}}}\\
f_{13}(k,t)&=&\frac{1+5\mathrm{i}}{2}\e^{2tk\e^{\frac{\pi\mathrm{i}}{6}}}+
\frac{1-5\mathrm{i}}{2}\e^{2tk\e^{\frac{-\pi\mathrm{i}}{6}}}\\
f_{21}(k,t)&=&-2\mathrm{i}\e^{2tk\e^{\frac{\pi\mathrm{i}}{6}}}+2\mathrm{i}\e^{2tk\e^{\frac{-\pi\mathrm{i}}{6}}}\\
f_{22}(k,t)&=&\frac{1-9\mathrm{i}}{2}\e^{2tk\e^{\frac{\pi\mathrm{i}}{6}}}+
\frac{1+9\mathrm{i}}{2}\e^{2tk\e^{\frac{-\pi\mathrm{i}}{6}}}\\
f_{23}(k,t)&=&-\mathrm{i}\e^{2tk\e^{\frac{\pi\mathrm{i}}{6}}}+\mathrm{i}\e^{2tk\e^{\frac{-\pi\mathrm{i}}{6}}}\\
f_{31}(k,t)&=&(-1-\mathrm{i})\e^{2tk\e^{\frac{\pi\mathrm{i}}{6}}}+
(-1+\mathrm{i})\e^{2tk\e^{\frac{-\pi\mathrm{i}}{6}}}\\
f_{32}(k,t)&=&\frac{-4-5\mathrm{i}}{2}\e^{2tk\e^{\frac{\pi\mathrm{i}}{6}}}+
\frac{-4+5\mathrm{i}}{2}\e^{2tk\e^{\frac{-\pi\mathrm{i}}{6}}}\\
f_{33}(k,t)&=&\frac{-1-\mathrm{i}}{2}\e^{2tk\e^{\frac{\pi\mathrm{i}}{6}}}+
\frac{-1+\mathrm{i}}{2}\e^{2tk\e^{\frac{-\pi\mathrm{i}}{6}}}
\end{eqnarray*}
\end{example}
\begin{example}
Let $x$ be a nonzero complex number ant let
$$E=\begin{pmatrix}
2\sqrt{3}-x-10&2\sqrt{3}-2x-23&\sqrt{3}-x-5\\4&\sqrt{3}+9&2\\
-2\sqrt{3}+2x+2&-4\sqrt{3}+4x+5&-\sqrt{3}+2x+1
\end{pmatrix}$$
Put
$$E(k)=
\begin{pmatrix}
e_{11}(k)&e_{12}(k)&e_{13}(k)\\e_{21}(k)&e_{22}(k)&e_{23}(k)\\e_{31}(k)&e_{32}(k)&e_{33}(k)
\end{pmatrix}$$
where
\begin{eqnarray*}
e_{11}(k)&=&2^{k+1}(\cos(\frac{k\pi}{6})-5\sin(\frac{k\pi}{6}))-x^{k}\\
e_{12}(k)&=&2^{k+1}(\cos(\frac{k\pi}{6})-\frac{23}{2}\sin(\frac{k\pi}{6}))-2x^{k}\\
e_{13}(k)&=&2^{k}(\cos(\frac{k\pi}{6})-5\sin(\frac{k\pi}{6}))-x^{k}\\
e_{21}(k)&=&2^{k+2}\sin(\frac{k\pi}{6})\\
e_{22}(k)&=&2^{k}(\cos(\frac{k\pi}{6})+9\sin(\frac{k\pi}{6}))\\
e_{23}(k)&=&2^{k+1}\sin(\frac{k\pi}{6})\\
e_{31}(k)&=&-2^{k+1}(\cos(\frac{k\pi}{6})-\sin(\frac{k\pi}{6}))+2x^{k}\\
e_{32}(k)&=&-2^{k+2}(\cos(\frac{k\pi}{6})-\frac{5}{4}\sin(\frac{k\pi}{6}))+4x^{k}\\
e_{33}(k)&=&-2^{k}(\cos(\frac{k\pi}{6})-\sin(\frac{k\pi}{6}))+2x^{k}
\end{eqnarray*}
We have, for all $k\geq1$,
$E^{k}=E(k)$. Then for all positive integer $k$
\begin{enumerate}[1.]
\item The eigenvalues of the matrix $E$ are $2\e^{\mathrm{i}\frac{\pi}{6}}, 2\e^{-\mathrm{i}\frac{\pi}{6}}$ and $x$.
\item Let $B$ be the matrix such that
\begin{eqnarray*}
B^{k}=
\begin{pmatrix}
g_{11}(k,t)&g_{12}(k,t)&g_{13}(k,t)\\g_{21}(k,t)&g_{22}(k,t)&g_{23}(k,t)\\
g_{31}(k,t)&g_{32}(k,t)&g_{33}(k,t)
\end{pmatrix}
\end{eqnarray*}
where
\begin{eqnarray*}
g_{11}(k)&=&(1+5\mathrm{i})(\ln(2)+\frac{\pi\mathrm{i}}{6})^{k}+(1-5\mathrm{i})(\ln(2)-\frac{\pi\mathrm{i}}{6})^{k}-\ln(x)^{k}\\
g_{12}(k)&=&(1+\frac{23}{2}\mathrm{i})(\ln(2)+\frac{\pi\mathrm{i}}{6})^{k}+
(1-\frac{23}{2}\mathrm{i})(\ln(2)-\frac{\pi\mathrm{i}}{6})^{k}-2\ln(x)^{k}\\
g_{13}(k)&=&\frac{1+5\mathrm{i}}{2}(\ln(2)+\frac{\pi\mathrm{i}}{6})^{k}+
\frac{1-5\mathrm{i}}{2}(\ln(2)-\frac{\pi\mathrm{i}}{6})^{k}-\ln(x)^{k}\\
g_{21}(k)&=&-2\mathrm{i}(\ln(2)+\frac{\pi\mathrm{i}}{6})^{k}+2\mathrm{i}(\ln(2)-\frac{\pi\mathrm{i}}{6})^{k}\\
g_{22}(k)&=&\frac{1-9\mathrm{i}}{2}(\ln(2)+\frac{\pi\mathrm{i}}{6})^{k}+
\frac{1+9\mathrm{i}}{2}(\ln(2)-\frac{\pi\mathrm{i}}{6})^{k}\\
g_{23}(k)&=&-\mathrm{i}(\ln(2)+\frac{\pi\mathrm{i}}{6})^{k}+\mathrm{i}(\ln(2)-\frac{\pi\mathrm{i}}{6})^{k}\\
g_{31}(k)&=&(-2-\mathrm{i})(\ln(2)+\frac{\pi\mathrm{i}}{6})^{k}+
(-2+\mathrm{i})(\ln(2)-\frac{\pi\mathrm{i}}{6})^{k}+2\ln(x)^{k}
\end{eqnarray*}
\begin{eqnarray*}
g_{32}(k)&=&\frac{-4-5\mathrm{i}}{2}(\ln(2)+\frac{\pi\mathrm{i}}{6})^{k}+
\frac{-4+5\mathrm{i}}{2}(\ln(2)-\frac{\pi\mathrm{i}}{6})^{k}+4\ln(x)^{k}\\
g_{33}(k)&=&\frac{-2-\mathrm{i}}{2}(\ln(2)+\frac{\pi\mathrm{i}}{6})^{k}+
\frac{-2+\mathrm{i}}{2}(\ln(2)-\frac{\pi\mathrm{i}}{6})^{k}+2\ln(x)^{k}
\end{eqnarray*}
Then $B$ is a logarithm of $E$.
\item The eigenvalues of the matrix $B$ are $\ln(2)+\frac{\pi\mathrm{i}}{6}$, $\ln(2)-\frac{\pi\mathrm{i}}{6}$ and $\ln(x)$.
\item Suppose that $x$ is not a negative real number. Then if we replace the logarithm $\ln(x)$ by the principal logarithm of $x$ in the above expressions, we obtain that $B$ is the principal logarithm of $E$.
\end{enumerate}
\end{example}
\begin{example}
Let
$$
C=\begin{pmatrix}1&3\\-3&-5
\end{pmatrix}$$
It is easily seen that for all nonnegative integer $k$
$$
C^{k}=(-2)^{k}\begin{pmatrix}\frac{-3k}{2}+1&\frac{-3k}{2}\\\frac{3k}{2}&\frac{3k}{2}+1
\end{pmatrix}$$
Then
$$
C(t)=\e^{(\ln(2)+\mathrm{i}\pi)t}\begin{pmatrix}\frac{-3t}{2}+1&\frac{-3t}{2}\\\frac{3t}{2}&\frac{3t}{2}+1
\end{pmatrix}$$
Thus
$$
C'(0)=\begin{pmatrix}\frac{-3}{2}+\ln(2)+\mathrm{i}\pi&\frac{-3}{2}\\\frac{3}{2}&\frac{3}{2}+\ln(2)+\mathrm{i}\pi
\end{pmatrix}$$
is a logarithm of $C$.
\\From Theorem~\ref{thm,22} we we have also that
$$
C'(0)^{k}=(\ln(2)+\mathrm{i}\pi)^{k-1}\begin{pmatrix}\frac{-3k}{2}+\ln(2)+\mathrm{i}\pi&\frac{-3k}{2}\\
\frac{3k}{2}&\frac{3k}{2}+\ln(2)+\mathrm{i}\pi
\end{pmatrix}$$
\end{example}

\end{document}